\documentclass[leqno]{article}
\usepackage{geometry}
\usepackage{graphicx}	
\usepackage[cp1251]{inputenc}
\usepackage[english]{babel}
\usepackage{mathtools}
\usepackage{amsfonts,amssymb,mathrsfs,amscd,amsmath,amsthm}
\usepackage{verbatim}

\usepackage{url}

\def\ig#1#2#3#4{\begin{figure}[!ht]\begin{center}%
\includegraphics[height=#2\textheight]{#1.eps}\caption{#4}\label{#3}%
\end{center}\end{figure}}

\def\thtext#1{
  \catcode`@=11
  \gdef\@thmcountersep{. #1}
  \catcode`@=12
}

\def\threst{
  \catcode`@=11
  \gdef\@thmcountersep{.}
  \catcode`@=12
}

\theoremstyle{plain}
\newtheorem{thm}{Theorem}[section]
\newtheorem{prop}[thm]{Proposition}
\newtheorem{cor}[thm]{Corollary}

\theoremstyle{definition}
\newtheorem{examp}[thm]{Example}
\newtheorem{rk}[thm]{Remark}

 \catcode`@=11
 \def\.{.\spacefactor\@m}
 \catcode`@=12

\def\R{\mathbb R}

\def\a{\alpha}
\def\b{\beta}
\def\e{\varepsilon}

\def\D{\Delta}

\def\l{\lambda}

\def\s{\sigma}

\def\0{\emptyset}
\def\:{\colon}
\def\<{\langle}
\def\>{\rangle}

\def\rom#1{\emph{#1}}
\def\({\rom(}
\def\){\rom)}
\def\sm{\setminus}
\def\ss{\subset}

\def\x{\times}

\def\diam{\operatorname{diam}}

\def\dis{\operatorname{dis}}

\def\cD{{\cal D}}

\def\cM{{\cal M}}

\def\cP{{\cal P}}
\def\cR{{\cal R}}

\begin{document}
\title{Gromov--Hausdorff Distance to Simplexes}
\author{D.\,S.~Grigor'ev, A.\,O.~Ivanov, A.\,A.~Tuzhilin}
\date{}
\maketitle

\begin{abstract} 
Geometric characteristics of metric spaces that appear in formulas of Gromov--Hausdorff distances from these spaces to so-called simplexes, i.e., to the metric spaces, all whose non-zero distances are the same are studied. The corresponding calculations essentially use geometry of partitions of these spaces. In the finite case, it gives the lengths of minimal spanning trees~\cite{TuzMST-GH}. In~\cite{IvaTuzSimpDist} a similar theory for compact metric spaces is worked out. Here we generalize the results from~\cite{IvaTuzSimpDist} to any bounded metric space, also, we simplify some proofs from~\cite{IvaTuzSimpDist}.
\end{abstract}

\section*{Introduction}
\markright{\thesection.~Introduction}
The ``space of spaces'' and ``space of subsets'' appear often in applications such as images comparison and recognition, besides, these spaces are important in various pure theoretical speculations, and so, they attract attention of various specialists for many years. One of the natural approaches to investigation of such spaces consists in introducing a distance function which measures the ``difference'' between the objects in consideration. In 1914 F.~Hausdorff~\cite{Hausdorff} defined a non-negative symmetric function on pairs of non-empty subsets of a metric space $X$: it equals the infimum of positive $r$ such that each subset from the pair belongs to the $r$-neighborhood of the other one. This function turns the family of all closed bounded subsets of $X$ into a metric space. Later D.~Edwards~\cite{Edwards} and, independently, M.~Gromov~\cite{Gromov} generalized this Hausdorff construction to the class of all metric spaces, in terms of isometric embeddings into another ``ambient'' metric spaces, see below. Now this function is called the Gromov--Hausdorff distance. Notice that the distance is symmetric and satisfies triangle inequality, however, it can be equal to infinity, and also can vanish for non-isometric metric spaces. Nevertheless, its restriction to the set $\cM$ of all compact metric spaces considered up to an isometry, satisfies all axioms of metric. The set $\cM$ endowed with the Gromov--Hausdorff distance is called the Gromov--Hausdorff space. The geometry of this metric space turns out to be rather exotic and is actively studied recently. It is well--known that  $\cM$ is linear connected, complete, separable, geodesic space, and that $\cM$ is not proper. A detailed introduction to geometry of the Gromov--Hausdorff space can be found in~\cite{BurBurIva}. In~\cite{INT} it was proved that $\cM$ is geodesic.

Actually, it is rather difficult to calculate the Gromov--Hausdorff distance between two given spaces. Even in the case of finite metric spaces there is no an effective algorithm, and brute forth enumeration based on the correspondences technic poorly suited even for spaces consisting of several dozen points. However, this technic, actively developing in recent times, allows to get various nontrivial theoretical results, see for example~\cite{INT}, \cite{ITAutVesnik}, or~\cite{Memoli}.

In the present paper we deal with calculation and estimation of the Gromov--Hausdorff distances from an arbitrary bounded metric space to so-called simplexes, namely, to the metric spaces, all whose non-zero distances are the same. In the case of finite spaces and simplexes, these calculations lead to a new interpretation of minimal spanning tree edges' lengths~\cite{TuzMST-GH}. In addition, these distances play an important role in investigation of the symmetry group of the space $\cM$, see~\cite{ITAutVesnik}. In~\cite{IvaTuzSimpDist} we have calculated and estimated the distances between finite simplexes and compact spaces. In particular, it gives us a possibility to construct two non-isometric finite metric spaces with equal distances to all finite simplexes.

In the present work we do not restrict ourselves neither with finite simplexes, nor with compact spaces. We define a few additional characteristics of bounded metric spaces and use them for exact formulas or exact estimations of the Gromov--Hausdorff distances between the spaces and simplexes in such general situation.

The work is partly supported by RFBR (Project~19-01-00775-a) and by President Program of Leading Scientific Schools Support (Project~NSh--6399.2018.1, Agreement No.~075--02--2018--867)

\section{Preliminaries}
\markright{\thesection.~Preliminaries}
Let $X$ be an arbitrary set.  By $\#X$ we denote the \emph{cardinality\/} of the set $X$.

Let $X$ be an arbitrary metric space. The distance between its points $x$ and $y$ we denote by $|xy|$. If $A,B\ss X$ are non-empty subsets of $X$, then we put $|AB|=\inf\bigl\{|ab|:a\in A,\,b\in B\bigr\}$. For $A=\{a\}$, we write $|aB|=|Ba|$ instead of $|\{a\}B|=|B\{a\}|$.

For each  point $x\in X$ and a number $r>0$, by $U_r(x)$ we denote the open ball with center $x$ and radius $r$; for any non-empty $A\ss X$ and a number $r>0$, we put $U_r(A)=\cup_{a\in A}U_r(a)$.

\subsection{Hausdorff and Gromov--Hausdorff Distances}
For non-empty $A,\,B\ss X$ put
$$
d_H(A,B)=\inf\bigl\{r>0:A\ss U_r(B),\ \text{and}\ B\ss U_r(A)\bigr\}=\max\{\sup_{a\in A}|aB|,\ \sup_{b\in B}|Ab|\}.
$$
This value is called the \emph{Hausdorff distance between $A$ and $B$}. It is well-known~\cite{BurBurIva} that the Hausdorff distance restricted to the set of all non-empty bounded closed subsets of $X$ is a metric.

Let $X$ and $Y$ be metric spaces. A triple $(X',Y',Z)$ consisting of a metric spaces $Z$, together with its subsets $X'$ and $Y'$ isometric to $X$ and $Y$, respectively, is called a \emph{realization of the pair $(X,Y)$}. The \emph{Gromov--Hausdorff distance $d_{GH}(X,Y)$ between $X$ and $Y$} is the infimum of real numbers $r$ such that there exists a realization  $(X',Y',Z)$ of the pair $(X,Y)$ with $d_H(X',Y')\le r$. It is well-known~\cite{BurBurIva} that $d_{GH}$ restricted to the set $\cM$ of all compact metric spaces considered up to an isometry, is a metric.

The following technic of correspondences is useful for calculation of the Gromov--Hausdorff distance.

Let $X$ and $Y$ be arbitrary non-empty sets. Recall that a \emph{relation\/} between the sets $X$ and $Y$ is a subset of the Cartesian product $X\x Y$.  By $\cP(X,Y)$ we denote the set of all \textbf{non-empty\/} relations between $X$ and $Y$. It is useful to consider each relation $\s\in\cP(X,Y)$ as a multivalued mapping, whose domain can be less than $X$. Then, similarly with the case of mappings, for each $x\in X$ and $A\ss X$, one can define their images
$$
\s(x)=\big\{y\in Y: (x,y)\in \s\big\}, \quad \text{and}\quad \s(A)=\bigcup_{a\in A}\s(a),
$$
and for any $y\in Y$ and $B\ss Y$, their preimages
$$
\s^{-1}(y)=\big\{x\in X: (x,y)\in R\big\}, \quad \text{and}\quad \s^{-1}(B)=\bigcup_{b\in B}\s^{-1}(b)
$$
are defined.

A relation $R\in\cP(X,Y)$ is called a \emph{correspondence} if $R(X)=Y$ and $R^{-1}(Y)=X$. The set of all correspondences between $X$ and $Y$ we denote by $\cR(X,Y)$.

Let $X$ and $Y$ be arbitrary metric spaces. The \emph{distortion $\dis\s$ of a relation $\s\in\cP(X,Y)$} is the value
$$
\dis\s=\sup\Bigl\{\bigl||xx'|-|yy'|\bigr|: (x,y),(x',y')\in\s\Bigr\}.
$$
It is easy to see that for any relations $\s_1,\s_2\in\cP(X,Y)$ such that $\s_1\ss\s_2$ it holds $\dis\s_1\le\dis\s_2$. In other words, the mapping $\dis\:\cP(X,Y)\to\R$ is monotone with respect to the partial ordering on $\cP(X,Y)$ generated by inclusion.

\begin{prop}[\cite{BurBurIva}]
For any metric spaces $X$ and $Y$ we have
$$
d_{GH}(X,Y)=\frac12\inf\bigl\{\dis R:R\in\cR(X,Y)\bigr\}.
$$
\end{prop}

A correspondence from $\cR(X,Y)$, that is minimal by inclusion is called \emph{irreducible}. By $\cR^0(X,Y)$ we denote  the set of all irreducible correspondences between $X$ and $Y$.

Notice (see~\cite{IvaTuzIrreducible}) that each irreducible correspondence $R\in\cR^0(X,Y)$ generates partitions $D^R_X$ and $D^R_Y$ of the spaces $X$ and $Y$, respectively, together with a bijection $f_R\:D^R_X\to D^R_Y$, such that
\begin{equation}\label{eq:irreducible}
R=\bigcup_{X_i\in D^R_X}X_i\x f_R(X_i),
\end{equation}
and, in addition, if $\#X_i>1$ then $\#f_R(X_i)=1$, and if $\#f_R(X_i)>1$ then $\#X_i=1$. Moreover, each bijection $f$ between arbitrary partitions $D_X$ and $D_Y$ of the spaces $X$ and $Y$, satisfying these properties, generates an irreducible correspondence by means of Formula~(\ref{eq:irreducible}).

\begin{prop}[\cite{IvaTuzIrreducible}]
For each $R\in\cR(X,Y)$ there exists an irreducible correspondence $R_0$ such that $R_0\ss R$. In particular, $\cR^0(X,Y)\ne\0$.
\end{prop}

Taking into account that the function $\dis$ is monotonic we get immediately the following result.

\begin{cor}\label{cor:GH-distance-irreducinle}
For any metric spaces $X$ and $Y$ we have
$$
d_{GH}(X,Y)=\frac12\inf\bigl\{\dis R:R\in\cR^0(X,Y)\bigr\}.
$$
\end{cor}

The correspondences can be used for simple proving the following well-known facts. For any metric space $X$ and a real number $\l>0$ by $\l X$ we denote  the metric space that differs from $X$ by multiplication of all its distances by $\l$.

\begin{prop}[\cite{BurBurIva}]\label{prop:GH_simple}
Let $X$ and $Y$ be metric spaces. Then
\begin{enumerate}
\item\label{prop:GH_simple:1} if $X$ is the single-point metric space, then $d_{GH}(X,Y)=\frac12\diam Y$\rom;
\item\label{prop:GH_simple:2} if $\diam X<\infty$, then
$$
d_{GH}(X,Y)\ge\frac12|\diam X-\diam Y|;
$$
\item\label{prop:GH_simple:3} $d_{GH}(X,Y)\le\frac12\max\{\diam X,\diam Y\}$, in particular,  $d_{GH}(X,Y)<\infty$ for bounded $X$ and $Y$\rom;
\item\label{prop:GH_simple:4} for any $X,Y\in\cM$ and any $\l>0$ we have $d_{GH}(\l X,\l Y)=\l d_{GH}(X,Y)$. Moreover, for $\l\ne1$, the unique space that remains the same under multiplication by $\l$ is the single-point space. In other words, the multiplication of metrics by $\l>0$ is a homothety of the space $\cM$ with the center at the single-point space.
\end{enumerate}
\end{prop}

\subsection{A Few Elementary Relations}
For calculations of the Gromov--Hausdorff distances one can use the following simple relations, whose proofs can be found in~\cite{IvaTuzSimpDist}.

\begin{prop}\label{prop:max_abs}
For any non-negative $a$ and $b$ it holds
$$
\max\big\{a,|b-a|\big\}\le\max\{a,b\}.
$$
\end{prop}

\begin{prop}\label{prop:many_abs_dif}
Let $A\ss\R$ be a non-empty bounded subset of the reals, and let $\l\in\R$. Then
$$
\sup_{a\in A}|\l-a|=\max\{\l-\inf A,\,\sup A-\l\}=\Big|\l-\frac{\inf A+\sup A}{2}\Big|+\frac{\sup A-\inf A}{2}.
$$
\end{prop}

\begin{prop}\label{prop:many_abs_dif_and_t}
Let $A\ss\R$ be a non-empty bounded subset of the reals, $\inf A\ge0$, and let $\l\in\R$. Then
$$
\sup_{a\in A}\big\{\l,|\l-a|\big\}=\max\{\l,\sup A-\l\}.
$$
\end{prop}

\begin{cor}\label{cor:many_abs_dif}
For any $a\ge 0$ and any $\l$ it holds
$$
\max\bigl\{\l,|a-\l|\bigr\}=\max\{\l,a-\l\}.
$$
\end{cor}

\section{Gromov--Hausdorff Distance between Bounded Metric Space and Simplex}
\markright{\thesection.~Distance between Bounded Metric Space and Simplex}

We call a metric space $X$ by a \emph{simplex}, if all its non-zero distances equal to each other. We denote by $\D$ a simplex, whose non-zero distances equal $1$. Thus, $\l\D$, $\l>0$, is a simplex, whose non-zero distances equal $\l$.

\subsection{Gromov--Hausdorff Distance to Simplexes with More Points}

The next result generalizes Theorem~4.1 from~\cite{IvaTuzSimpDist}.

\begin{thm}\label{thm:dist-n-simplex-bigger-dim}
Let $X$ be an arbitrary bounded metric space, and $\#X<\#\l\D$, then
$$
2d_{GH}(\l\D,X)=\max\{\l,\diam X-\l\}.
$$
\end{thm}

\begin{proof}
If $\#X=1$, then $\diam X=0$, and, by Proposition~\ref{prop:GH_simple}, we have
$$
2d_{GH}(\l\D,X)=\diam\l\D=\l=\max\{\l,\diam X-\l\}.
$$

Now, let $\#X>1$. Choose an arbitrary $R\in\cR(\l\D,X)$. Since $\#X<\#\l\D$, then there exists $x\in X$ such that $\#R^{-1}(x)\ge2$, thus, $\dis R\ge\l$ and $2d_{GH}(\l\D,X)\ge\l$.

Consider an arbitrary sequence $(x_i,y_i)\in X\x X$ such that $|x_iy_i|\to\diam X$. If it contains a subsequence $(x_{i_k},y_{i_k})$ such that for each $i_k$ there exists $z\in\l\D$, $(z,x_{i_k})\in R$, $(z,y_{i_k})\in R$, then $\dis R\ge\diam X$ and $2d_{GH}(\l\D,X)\ge\max\{\l,\diam X\}$.

If such subsequence does not exist, then there exists a subsequence $(x_{i_k},y_{i_k})$ such that for any  $i_k$ there exist distinct $z_k,w_k\in\l\D$, $(z_k,x_{i_k})\in R$, $(w_k,y_{i_k})\in R$, and, therefore, $2d_{GH}(\l\D,X)\ge\max\bigl\{\l,|\diam X-\l|\bigr\}$.

By Proposition~\ref{prop:max_abs},
$$
\max\{\l,\diam X\}\ge\max\bigl\{\l,|\diam X-\l|\bigr\},
$$
thus, in the both cases we have $2d_{GH}(\l\D,X)\ge\max\bigl\{\l,|\diam X-\l|\bigr\}$.

Choose an arbitrary $x_0\in X$, then, by assumption, $\#X>1$, and, thus, the set $X\sm\{x_0\}$ is not empty. Since $\#X<\#\l\D$, then $\l\D$ contains a subset $\l\D'$ of the same cardinality with $X\sm\{x_0\}$. Let $g\:\l\D'\to X\sm\{x_0\}$ be an arbitrary bijection, and $\l\D''=\l\D\sm\l\D'$, then $\l\D''\ne\0$. Consider the following correspondence:
$$
R_0=\Bigl\{\bigl(z',g(z')\bigr):z'\in\l\D'\Bigr\}\cup\bigl(\{x_0\}\x\l\D''\bigr).
$$
Then $\dis R_0\le\max\bigl\{\l,|\diam X-\l|\bigr\}$, and this implies
$$
2d_{GH}(\l\D,X)=\max\bigl\{\l,|\diam X-\l|\bigr\}.
$$
It remains to apply Corollary~\ref{cor:many_abs_dif}.
\end{proof}

\subsection{Gromov--Hausdorff Distance to Simplexes with at most the Same Number of Points}

Let $X$ be an arbitrary set and $m$ a cardinal number, $m\le\#X$.  By $\cD_m(X)$ we denote the set of all partitions of $X$ into $m$ non-empty subsets.

Now, let $X$ be a metric space, then for each $D=\{X_i\}_{i\in I}\in\cD_m(X)$ we put
$$
\diam D=\sup_{i\in I}\diam X_i.
$$
Further, for any non-empty $A,B\ss X$ let
$$
|AB|=\inf\bigl\{|ab|:(a,b)\in A\x B\bigr\},\quad \text{and}\quad |AB|'=\sup\bigl\{|ab|:(a,b)\in A\x B\bigr\},
$$
and for each $D=\{X_i\}_{i\in I}\in\cD_m(X)$ we define
$$
\a(D)=\inf\bigl\{|X_iX_j|:i\ne j\bigr\},\quad \text{and}\quad \b(D)=\sup\bigl\{|X_iX_j|':i\ne j\bigr\}.
$$

Let $\l\D$ be a simplex of cardinality $m$. Choose an arbitrary $D\in\cD_m(X)$, any bijection $g\:\l\D\to D$, and construct the correspondence $R_D\in\cR(\l\D,X)$ in the following way:
$$
R_D=\bigcup_{z\in\l\D}\{z\}\x g(z).
$$
Clearly that each correspondence $R_D$ is irreducible.

The next result gives a natural generalization of Proposition~4.5 from~\cite{IvaTuzSimpDist}.

\begin{prop}\label{prop:disRD}
Let $X$ be an arbitrary bounded metric space and $m=\#\l\D\le\#X$. Then for any $D\in\cD_m(X)$ it holds
$$
\dis R_D=\max\{\diam D,\,\l-\a(D),\,\b(D)-\l\}.
$$
\end{prop}

\begin{cor}\label{cor:disRD}
Let $X$ be an arbitrary bounded metric space and $m=\#\l\D\le\#X$. Then for any $D\in\cD_m(X)$ we have
$$
\dis R_D=\max\{\diam D,\,\l-\a(D),\,\diam X-\l\}.
$$
\end{cor}

\begin{proof}
Notice that $\diam D\le\diam X$ and $\b(D)\le\diam X$. In addition, if $\diam D<\diam X$, and $(x_i,y_i)\in X\x X$ is a sequence such that $|x_iy_i|\to\diam X$, then, starting from some $i$, the points $x_i$ and $y_i$ belong to different elements of $D$, therefore, in this case we have $\b(D)=\diam X$, and the formula is proved.

Now, let $\diam D=\diam X$, then $\b(D)-\l\le\diam X$ and $\diam X-\l\le\diam X$, thus
\begin{multline*}
\max\{\diam D,\,\l-\a(D),\,\b(D)-\l\}=\max\{\diam X,\,\l-\a(D)\}=\\ =\max\{\diam D,\,\l-\a(D),\,\diam X-\l\},
\end{multline*}
that completes the proof.
\end{proof}

Let us formulate and prove an analogue of Proposition~4.6 from~\cite{IvaTuzSimpDist}, without use of Theorem~4.3 from~\cite{IvaTuzSimpDist}.

\begin{prop}\label{prop:GH-dist-RD}
Let $X$ be an arbitrary bounded metric space and $m=\#\l\D\le\#X$. Then
$$
2d_{GH}(\l\D,X)=\inf_{D\in\cD_m(X)}\dis R_D.
$$
\end{prop}

\begin{proof}
By Corollary~\ref{cor:GH-distance-irreducinle}, $2d_{GH}(\l\D,X)=\inf_{R\in\cR^0(\l\D,X)}\dis R$, thus it suffice to prove that for any irreducible correspondence $R\in\cR^0(\l\D,X)$ there exists $D\in\cD_m(X)$ such that $\dis R_D\le\dis R$.

Let us choose an arbitrary $R\in\cR^0(\l\D,X)$ such that it cannot be represented in the form $R_D$, then the partition $D^R_{\l\D}$ is not pointwise, i.e., there exists $x\in X$ such that $\#R^{-1}(x)\ge 2$, therefore, $\dis R\ge\l$.

Define a metric on the set $D^R_{\l\D}$ to be equal $\l$ between any its distinct elements, then this metric space is isometric to a  simplex $\l\D'$. The correspondence $R$ generates naturally  another correspondence $R'\in\cR(\l\D',X)$, namely, if $D^R_{\l\D}=\{\D_i\}_{i\in I}$, and $f_R\:D^R_{\l\D}\to D^R_X$ is the bijection generated by $R$, then
$$
R'=\bigcup_{i\in I}\{\D_i\}\x f_R(\D_i).
$$
It is easy to see that $\dis R=\max\{\l,\,\dis R'\}$. Moreover, $R'$ is generated by the partition $D'=D^R_X$, i.e., $R'=R_{D'}$, thus, by Corollary~\ref{cor:disRD}, we have
$$
\dis R'=\max\{\diam D',\,\l-\a(D'),\,\diam X-\l\},
$$
and hence,
$$
\dis R=\max\{\l, \diam D',\,\l-\a(D'),\,\diam X-\l\}=\max\{\l,\,\diam D',\,\diam X-\l\}.
$$
Since $\#D'\le m$, the partition $D'$ has a subpartition $D\in\cD_m(X)$. Clearly, $\diam D\le\diam D'$, therefore,
$$
\dis R_D=\max\{\diam D,\,\l-\a(D),\,\diam X-\l\}\le\max\{\diam D',\,\l,\,\diam X-\l\}=\dis R,
$$
q.e.d.
\end{proof}

\begin{cor}\label{cor:GH-dist-alpha-beta}
Let $X$ be an arbitrary bounded metric space and $m=\#\l\D\le\#X$. Then
$$
2d_{GH}(\l\D,X)=\inf_{D\in\cD_m(X)}\max\{\diam D,\,\l-\a(D),\,\diam X-\l\}.
$$
\end{cor}

For any metric space $X$ put
$$
\e(X)=\inf\bigl\{|xy|:x,y\in X,\,x\ne y\bigr\}.
$$
Notice that $\e(X)\le\diam X$, and for a bounded $X$ the equality holds, iff $X$ is a simplex.

Corollary~\ref{cor:GH-dist-alpha-beta} immediately implies the following result that is proved in~\cite{IvaTuzSimpDist}.

\begin{thm}[\cite{IvaTuzSimpDist}]\label{thm:dist-n-simplex-same-dim}
Let $X$ be a finite metric space and $\#\l\D=\#X$, then
$$
2d_{GH}(\l\D,X)=\max\bigl\{\l-\e(X),\,\diam X-\l\bigr\}.
$$
\end{thm}

\subsection{New Exact Formulas and Estimates}

Pass to our main new results. For an arbitrary metric space $X$, $m=\#X$, put
\begin{gather*}
\a_m^-(X)=\inf_{D\in\cD_m(X)}\a(D),\qquad \a_m(X)=\a_m^+(X)=\sup_{D\in\cD_m(X)}\a(D),\\
d_m(X)=d_m^-(X)=\inf_{D\in\cD_m(X)}\diam D,\qquad d_m^+(X)=\sup_{D\in\cD_m(X)}\diam D.
\end{gather*}

\begin{rk}
We introduced simpler notations for $\a_m^+(X)$ and $d_m^-(X)$, because these values, in contrast to their ``twins'' $\a_m^-(X)$ and $d_m^+(X)$, appear much more often in  the formulas below.
\end{rk}

Notice that $\a_m(X)=0$, iff for each $D\in\cD_m(X)$ the equality $\a(D)=0$ holds.

\begin{examp}\label{examp:inf-card}
Let $X$ be an infinite compact metric space, and $m$ an infinite cardinal number, $m\le\#X$. Then $\a_m(X)=0$.

Indeed, choose an arbitrary $D=\{X_i\}_{i\in I}\in\cD_m(X)$, and pick up a point $x_i\in X_i$ in each $X_i$. Since $X$ is compact, the set $\{x_i\}_{i\in I}$ contains a convergent subsequence $\{x_{i_k}\}$ consisting of pairwise distinct points. Therefore, $|X_{i_k}X_{i_{k+1}}|\to0$ as $k\to\infty$, thus $\a(D)=0$.
\end{examp}

\begin{examp}
Let $X$ be a bounded infinite metric space represented as a disjoint union of $n$ infinite compact spaces, and let $m$ be an infinite cardinal number, $n<m\le\#X$, then $\a_m(X)=0$.

Indeed, let $\{Z_j\}_{j\in J}$, $\#J=n$, be a partition of the space $X$ into $n$ compact subsets $Z_j$. Choose an arbitrary $D=\{X_i\}_{i\in I}\in\cD_m(X)$. Since $n<m$, then for some $j\in J$ the family of all non-empty intersections $X_i\cap Z_j$ forms an infinite partition of the compact subset $Z_j$. It remains to apply the reasoning from Example~\ref{examp:inf-card}.
\end{examp}

\begin{examp}
Let $X$ be a connected bounded metric space, and let $m$ be a cardinal number, $2\le m\le\#X$. Then $\a_m(X)=0$.

Indeed, choose an arbitrary $D=\{X_i\}_{i\in I}\in\cD_m(X)$, any $X_i$, and put $X'_i=\cup_{j\ne i\in I}X_j$, then $D'=\{X_i,X'_i\}\in\cD_2(X)$. Since the space $X$ is connected, then $|X_iX'_i|=0$, thus there exists a sequence $X_{j_k}$, $j_k\ne i$, such that $|X_iX_{j_k}|\to0$ as $k\to\infty$, and hence, $\a(D)=0$.
\end{examp}

The following example can be considered similarly.

\begin{examp}
Let $X$ be an arbitrary bounded metric space consisting of $n$ connected components, and let $m$ be a cardinal number, $n<m\le\#X$, then $\a_m(X)=0$.
\end{examp}

Corollary~\ref{cor:GH-dist-alpha-beta} implies the following result.

\begin{thm}\label{thm:dist-n-simplex-same-dim-alpha0}
Let $X$ be an arbitrary bounded metric space, $m=\#\l\D\le\#X$, and $\a_m(X)=0$, then
$$
2d_{GH}(\l\D,X)=\max\bigl\{d_m(X),\,\l,\,\diam X-\l\bigr\}.
$$
\end{thm}

Consider the graph of dependence of the value $2d_{GH}(\l\D,X)$ on $\l$ in accordance with  Theorem~\ref{thm:dist-n-simplex-same-dim-alpha0}, see Figure~\ref{fig:alpha0}.

\ig{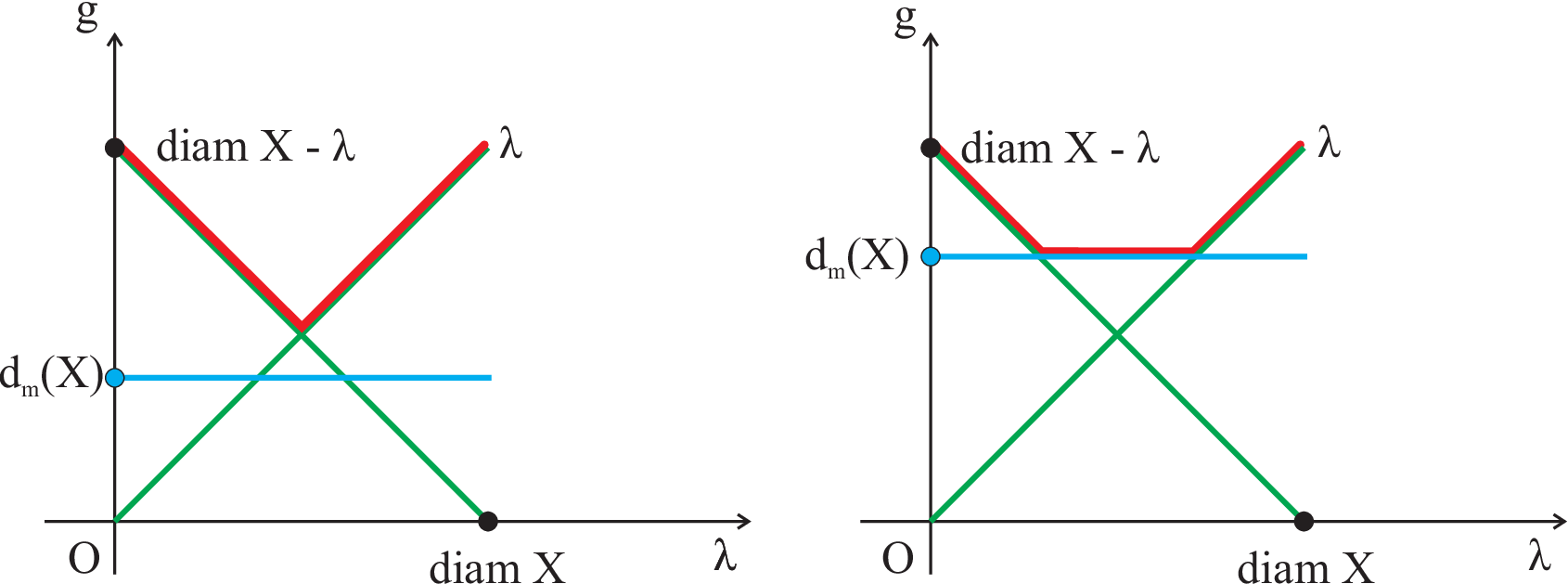}{0.27}{fig:alpha0}{The graph of the function $g(\l)=2d_{GH}(\l\D,X)$ for $\a_m(X)=0$ and various values of $d_m(X)$.}

Figure~\ref{fig:alpha0} leads naturally to the following result.

\begin{cor}
Let $X$ --- be an arbitrary bounded metric space, $m=\#\l\D\le\#X$, and $\a_m(X)=0$, then
\begin{enumerate}
\item if $d_m(X)\le\frac12\diam X$, then
$$
2d_{GH}(\l\D,X)=
\begin{cases}
\diam X-\l,&\text{for}\ \ \l\le\frac12\diam X,\\
\hfil\l,&\text{for}\ \ \l\ge\frac12\diam X;
\end{cases}
$$
\item if $d_m(X)>\frac12\diam X$, then
$$
2d_{GH}(\l\D,X)=
\begin{cases}
\diam X-\l,&\text{for}\ \ \l\le\diam X-d_m(X),\\
\hfil d_m(X),&\text{for}\ \ \diam X-d_m(X)\le\l\le d_m(X),\\
\hfil\l,&\text{for}\ \ \l\ge d_m(X).
\end{cases}
$$
\end{enumerate}
\end{cor}

Another particular case corresponds to the condition $d_m(X)=\diam X$, that is equivalent to the one $\diam D=\diam X$ for all $D\in\cD_m(X)$.

\begin{examp}
For any simplex $X=\l\D$, any cardinal number $0<m<\#\l\D$, and any $D\in\cD_m(X)$ it holds $\diam D=\l=\diam X$, thus $d_m(X)=\diam X$.
\end{examp}

\begin{examp}
Let $X=S^1$ be the standard unit circle in the Euclidean plane, and $m=2$. Then $d_m(X)=\diam X=2$.

Indeed, suppose the contrary, i.e\., there exists $D=\{X_1,X_2\}\in\cD_m(X)$ such that $\diam D<2$. For $x\in X$ by $x'$ we denote  the point of the circle $X$ that is diametrically opposite to $x$. The above assumption implies that for each $x\in X_1$ the point $x'$ belongs to $X_2$, and, therefore, some open neighborhood of $x$ belongs to $X_1$ too. Thus, $X_1$ and, similarly, $X_2$ are non-empty open subsets of the circle $X$ forming a partition of $X$, but this contradicts to connectedness of $X$.
\end{examp}

\begin{cor}
Let $X$ be an arbitrary bounded metric space, $m=\#\l\D\le\#X$, and $d_m(X)=\diam X$, then
$$
2d_{GH}(\l\D,X)=\max\bigl\{\diam X,\,\l-\a_m(X)\bigr\}.
$$
In particular,
\begin{enumerate}
\item for $\l\le\diam X+\a_m(X)$ it holds $2d_{GH}(\l\D,X)=\diam X$\rom;
\item for $\l\ge\diam X+\a_m(X)$ it holds $2d_{GH}(\l\D,X)=\l-\a_m(X)$.
\end{enumerate}
\end{cor}

Now, consider the general case. It turns out that one can also obtain some exact formulas, but not for all $\l$. By $A$ we denote the intersection point of the graphs of the functions $\diam X-\l$, $\l\in[0,\diam X]$, and $\l-\a_m^-(X)$, $\l\ge\a_m^-(X)$. Consider all possible locations  of the point $A$ with respect to the horizontal strip $S$ between $d_m(X)$ and $d_m^+(X)$.

{\bf (1) The point $A$ lies below the strip $S$,} see  Figure~\ref{fig:general1}.

\ig{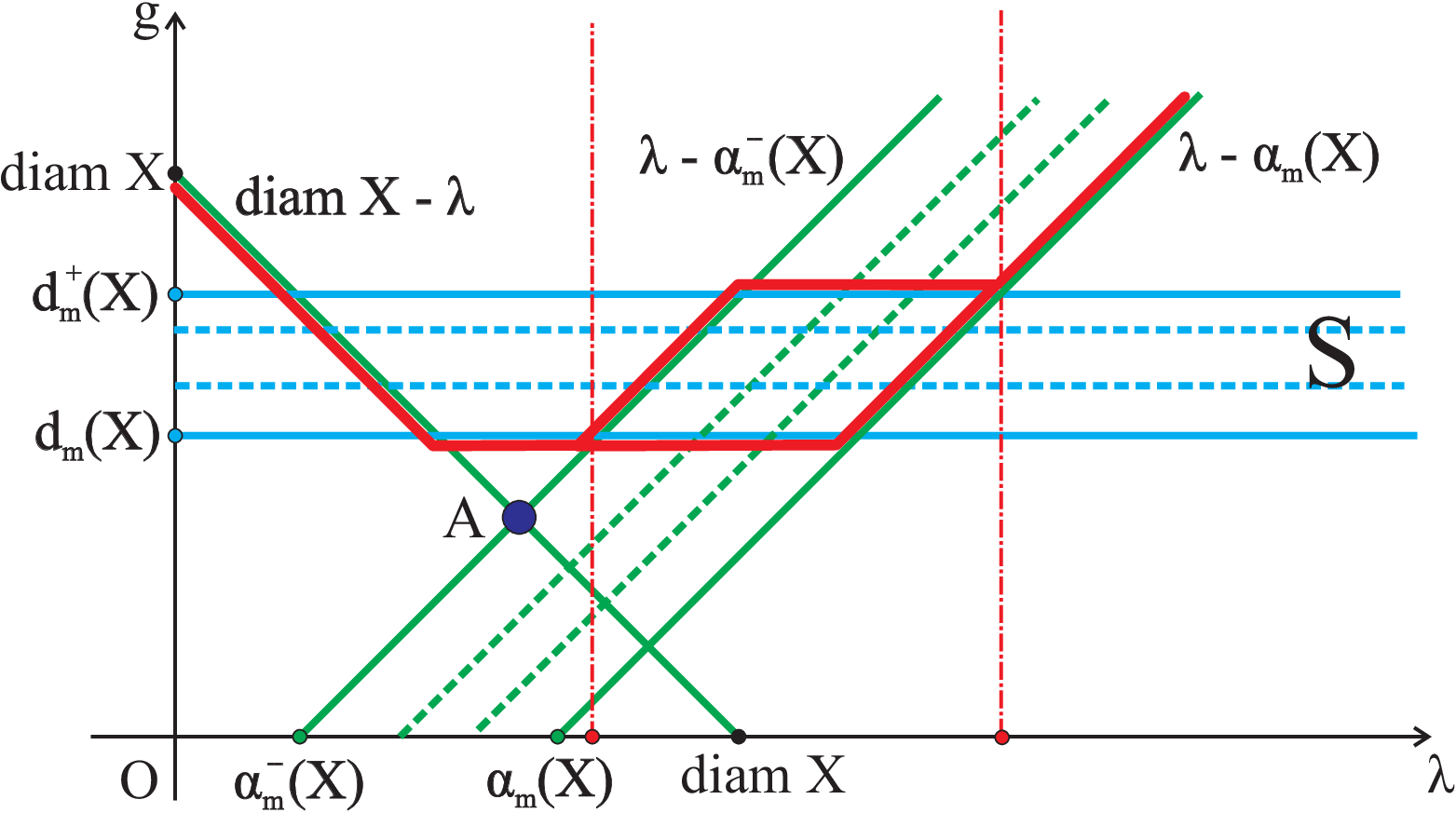}{0.27}{fig:general1}{The graph of the function $g(\l)=2d_{GH}(\l\D,X)$, the point $A$ lies below the strip $S$.}

The vertical dashed lines partitions the figure into three parts: left, middle, and right. In the left and right parts the bold lines represent the exact values of the function $g$. In the middle part the bold line bounds a rhombus that contains all points $\bigl(\a(D),\diam D\bigr)$, thus the ``top and left parts'' of the rhombus boundary gives the upper bound for the function $g$, and the ``bottom and right parts'' gives the lower bound.

After calculation of the coordinates of the graphs intersection points, we get the following result.

\begin{cor}
Let $X$ be an arbitrary bounded metric space, $m=\#\l\D\le\#X$. Suppose that $\diam X-\a_m^-(X)\le2d_m(X)$, then
\begin{itemize}
\item if $\l\le\a_m^-(X)+d_m(X)$, then
$$
2d_{GH}(\l\D,X)=\max\bigl\{\diam X-\l,d_m(X)\bigr\};
$$
\item if $\a_m^-(X)+d_m(X)\le\l\le\a_m(X)+d_m^+(X)$, then
$$
\max\bigl\{\l-\a_m(X),d_m(X)\bigr\}\le2d_{GH}(\l\D,X)\le\min\bigl\{\l-\a_m^-(X),d_m^+(X)\bigr\};
$$
\item if $\l\ge\a_m(X)+d_m^+(X)$, then
$$
2d_{GH}(\l\D,X)=\l-\a_m(X).
$$
\end{itemize}
\end{cor}

{\bf (2) The point $A$ belongs to the strip $S$}, see Figure~\ref{fig:general2}.

\ig{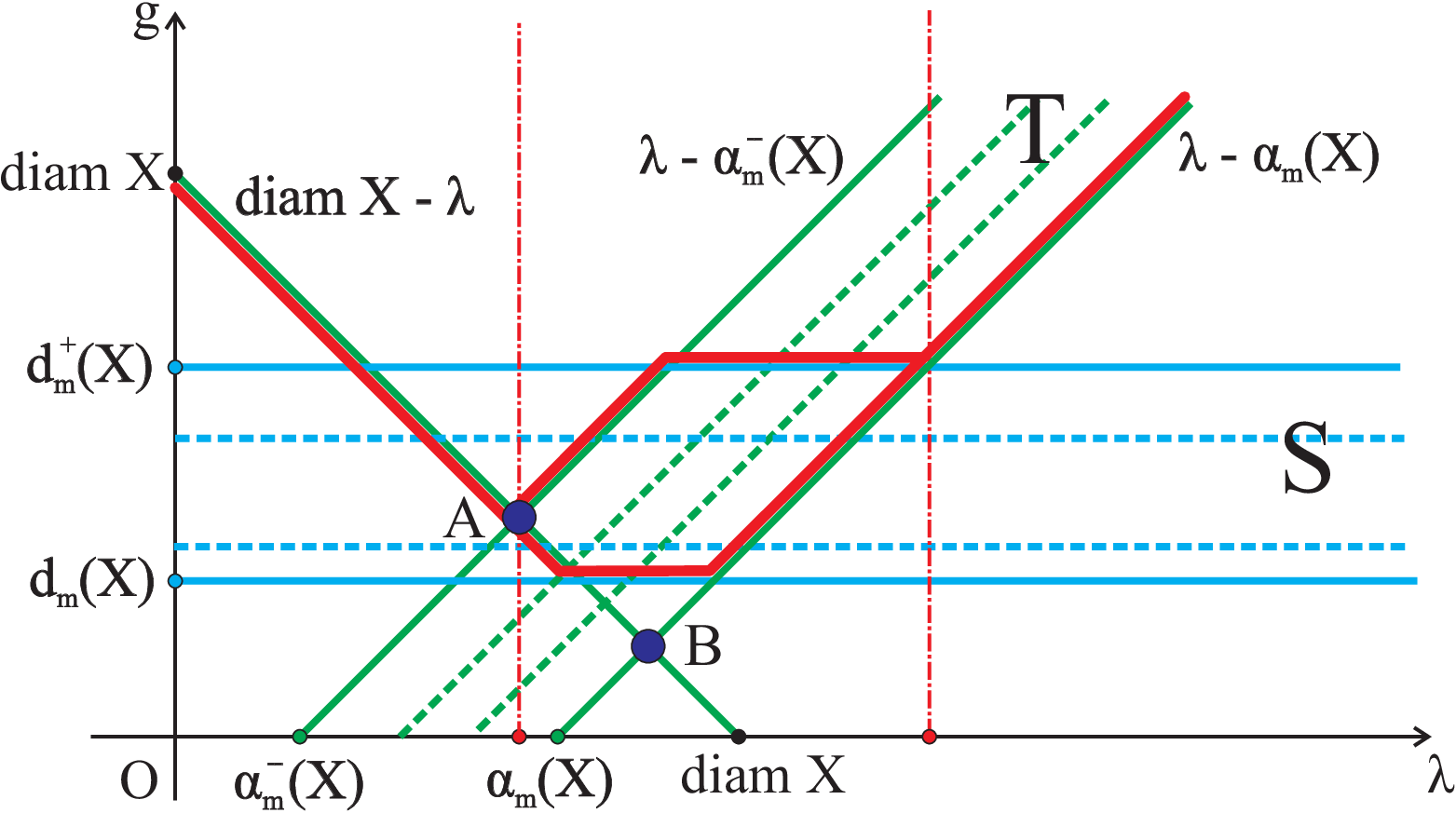}{0.27}{fig:general2}{The graph of the function $g(\l)=2d_{GH}(\l\D,X)$, the point $A$ belongs to the strip $S$.}

Let $B$ be the intersection point of the graphs of the functions $\diam X-\l$ and $\l-\a_m(X)$.  Again, the vertical dashed lines partition the figure into three parts, and the marginal parts contain the graphs of exact values of the function $g$. In the middle part the bold line bounds a $5$-gone, that degenerates into a $4$-gon, when the point $B$ gets to the strip $S$: although all points $\bigl(\a(D),\diam D\bigr)$ belong to the rhombus obtained as intersection of the strip $S$ and the slope strip $T$ between the graphs of the functions $\l-\a_m^-(X)$ and $\l-\a_m(X)$, the value of the function $g$ cannot be less than $\diam X-\l$, and the graph of the latter function cuts from the rhombus the corresponding domain (a triangle, when $B$ belongs to the interior of the strip $S$, and $4$-gon otherwise). 

After calculation of the coordinates of the graphs intersection points, we get the following result.

\begin{cor}
Let $X$ be an arbitrary bounded metric space, $m=\#\l\D\le\#X$. Suppose that $2d_m(X)\le\diam X-\a_m^-(X)\le2d_m^+(X)$, then
\begin{itemize}
\item if $\l\le\frac12\bigl(\a_m^-(X)+\diam(X)\bigr)$, then
$$
2d_{GH}(\l\D,X)=\diam X-\l;
$$
\item if $\frac12\bigl(\a_m^-(X)+\diam(X)\bigr)\le\l\le\a_m(X)+d_m^+(X)$, then
$$
\max\bigl\{\diam X-\l,\l-\a_m(X),d_m(X)\bigr\}\le2d_{GH}(\l\D,X)\le\min\bigl\{\l-\a_m^-(X),d_m^+(X)\bigr\};
$$
\item if $\l\ge\a_m(X)+d_m^+(X)$, then
$$
2d_{GH}(\l\D,X)=\l-\a_m(X).
$$
\end{itemize}
\end{cor}

{\bf (3) The point $A$ lies above the strip $S$.} There are two subcases, depending on location of the point $B$.

{\bf (3.1) The upper boundary line of the strip $S$ lies between the points $A$ and $B$}, see  Figure~\ref{fig:general3-1}.

\ig{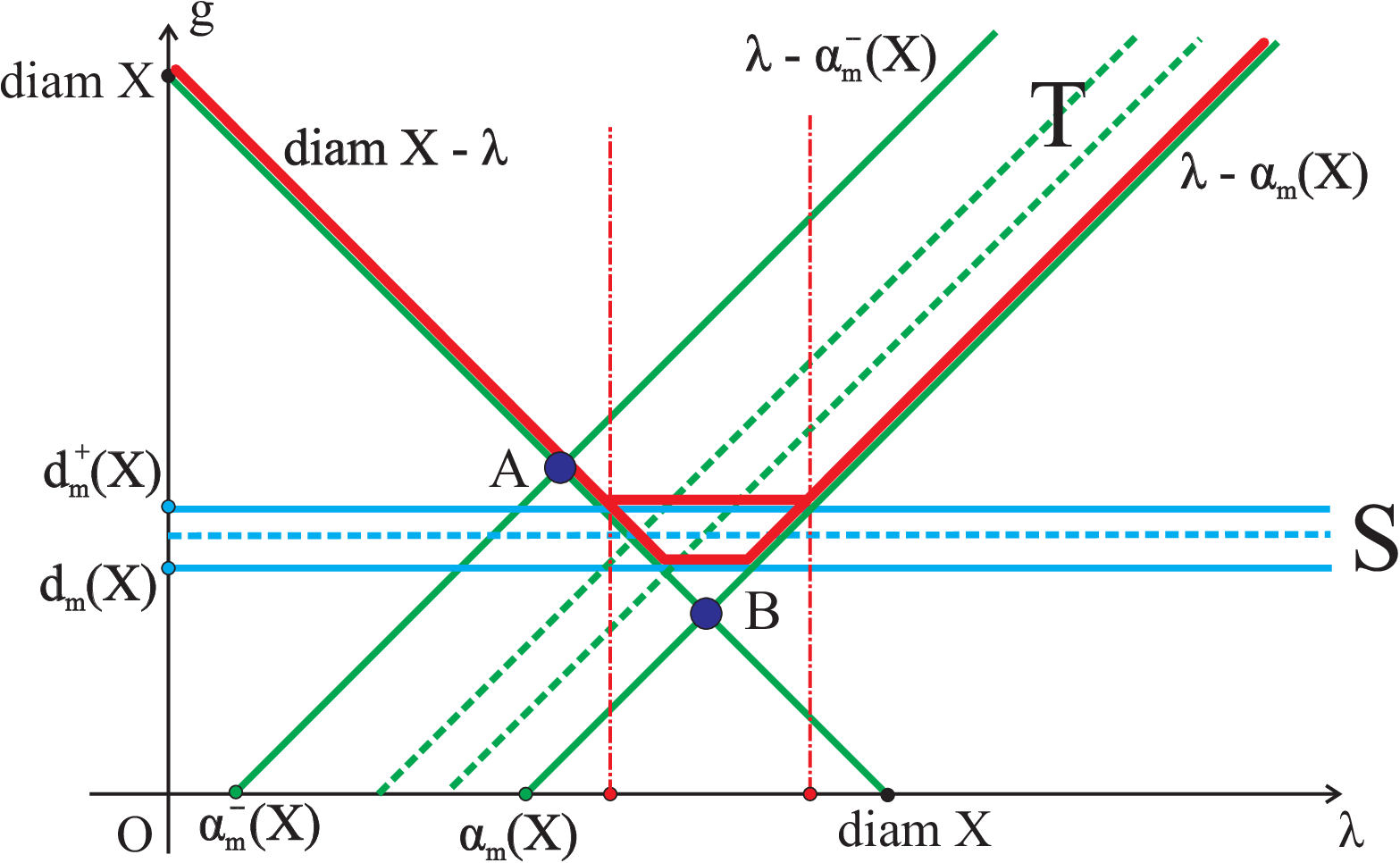}{0.28}{fig:general3-1}{The graph of the function $g(\l)=2d_{GH}(\l\D,X)$, the upper boundary line of the strip $S$ lies between the points $A$ and $B$.} 

In this case, the ``indeterminacy domain'' in the middle part of the figure forms a trapezium, if $B$ lies strictly below the strip $S$, and this trapezium degenerates into a triangle, providing $B$ belongs to this strip. 

After calculation of the coordinates of the graphs intersection points, we get the following result.

\begin{cor}
Let $X$ be an arbitrary bounded metric space, $m=\#\l\D\le\#X$. Suppose that
$\diam X-\a_m(X)\le2d_m^+(X)\le\diam X-\a_m^-(X)$, then
\begin{itemize}
\item if $\l\le\diam(X)-d_m^+(X)$, then
$$
2d_{GH}(\l\D,X)=\diam X-\l;
$$
\item if $\diam(X)-d_m^+(X)\le\l\le\a_m(X)+d_m^+(X)$, then
$$
\max\bigl\{\diam X-\l,\l-\a_m(X),d_m(X)\bigr\}\le 2d_{GH}(\l\D,X) \le d_m^+(X);
$$
\item if $\l\ge\a_m(X)+d_m^+(X)$, then
$$
2d_{GH}(\l\D,X)=\l-\a_m(X).
$$
\end{itemize}
\end{cor}

{\bf (3.2) The point $B$ lies above the strip $S$}, see Figure~\ref{fig:general3-2}.

\ig{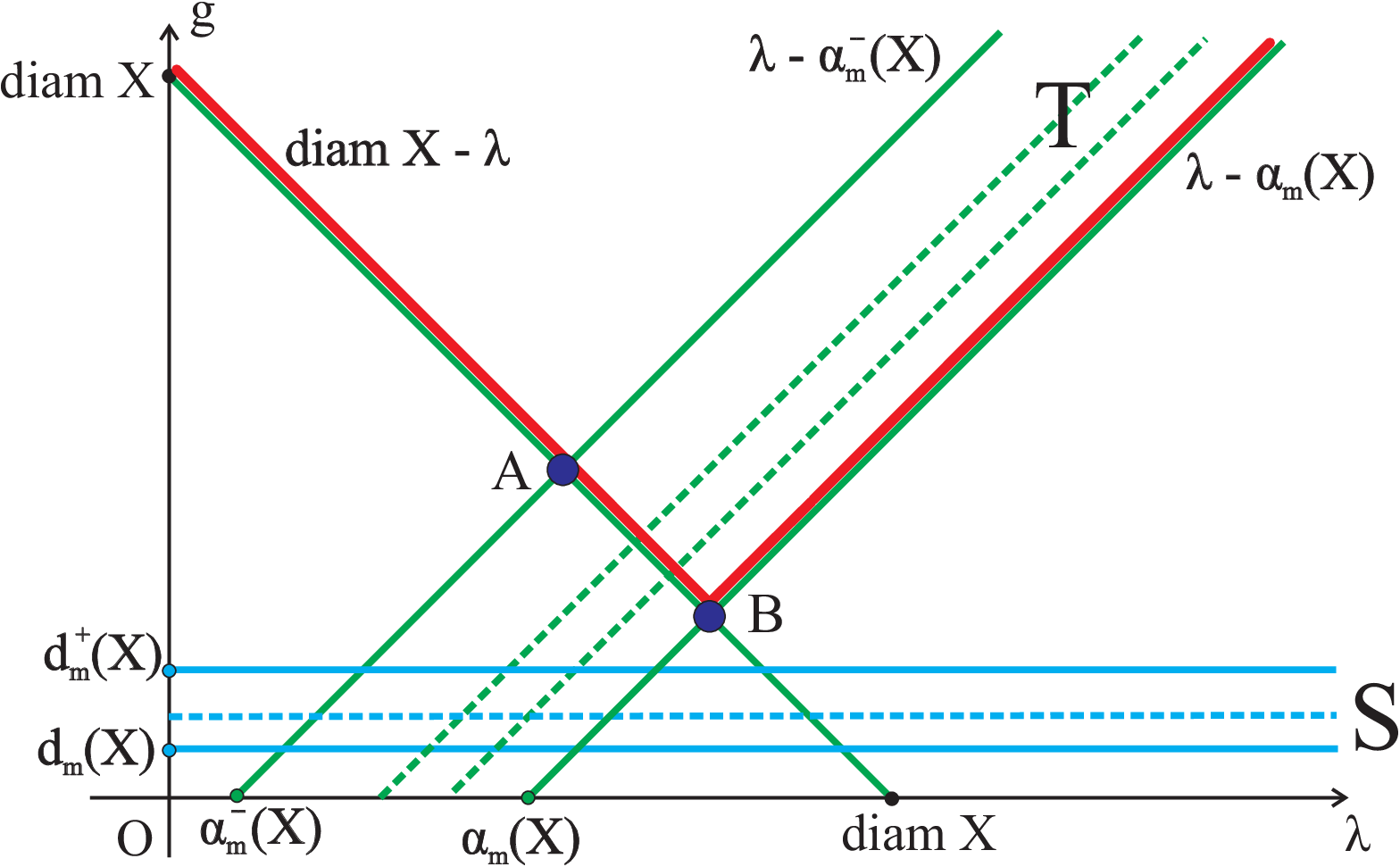}{0.28}{fig:general3-2}{The graph of the function $g(\l)=2d_{GH}(\l\D,X)$, the point $B$ lies above the strip $S$.}

In this case, the function $g$ can be calculated exactly  for all $\l$.

\begin{cor}
Let $X$ be an arbitrary bounded metric space, $m=\#\l\D\le\#X$. Suppose that $2d_m^+(X)\le\diam X-\a_m(X)$, then
$$
2d_{GH}(\l\D,X)=\max\bigl\{\diam X-\l,\l-\a_m(X)\bigr\}.
$$
\end{cor}

\end{document}